\documentclass[11pt,a4paper]{article}
\usepackage{latexsym,bm}
\usepackage{mathrsfs}
\usepackage{amsmath,amsthm}
\usepackage{graphicx}
\usepackage{amssymb}
\usepackage{CJK}
\usepackage{color}
\usepackage{cite}
\usepackage{array}
\usepackage{stfloats}
\usepackage[colorlinks,
            linkcolor=blue,       
            anchorcolor=blue,  
            citecolor=blue,        
            ]{hyperref}
\usepackage{caption}
\usepackage{graphicx}
\usepackage{epstopdf}
\usepackage{upgreek}
\usepackage{tikz}

\theoremstyle{plain}
\newtheorem{thm}{Theorem}[section]
\newtheorem{cor}[thm]{Corollary}
\newtheorem{lem}[thm]{Lemma}

\theoremstyle{definition}
\newtheorem{defn}[thm]{Definition}
\theoremstyle{plain}

\theoremstyle{problem}

\theoremstyle{Question}
\newtheorem{ques}{Question}

\theoremstyle{plain}

\theoremstyle{plain}
\newtheorem{cla}{Claim}
\theoremstyle{plain}

\theoremstyle{plain}

\DeclareGraphicsExtensions{.eps,.eps.gz}
\usepackage[top=2.5cm,bottom=2.5cm,left=2.8cm,right=2.2cm]{geometry}
\normalsize \rm
\allowdisplaybreaks[4]
\makeatletter
\@addtoreset{equation}{section}
\makeatother

 \usepackage{indentfirst}
\begin{document}
\begin{CJK}{GBK}{song}
\newcommand{\song}{\CJKfamily{song}}    
\newcommand{\fs}{\CJKfamily{fs}}        
\newcommand{\kai}{\CJKfamily{kai}}      
\newcommand{\hei}{\CJKfamily{hei}}      
\newcommand{\li}{\CJKfamily{li}}        
\renewcommand\figurename{Fig.}

\begin{center}
{{\huge  Combinatorial explanation of the weighted Kirchhoff index of graphs }} \\[18pt]
{\Large Wensheng Sun$^{1}$,  Yujun Yang$^{2}$,    Shou-Jun Xu$^{1,*}$  \footnotetext{*Corresponding author\\ \noindent E-mail addresses: wensheng07002@163.com(W. Sun),  yangyj@yahoo.com(Y. Yang), shjxu@lzu.edu.cn(S.-J. Xu),    }}\\[6pt]
{ \footnotesize  $^{1}$ School of Mathematics and Statistics, Gansu Center for Applied Mathematics, Lanzhou University, Lanzhou, Gansu 730000 China\\
$^{2}$ Department of Mathematics and Artificial Intelligence, Qilu University of Technology (Shandong Academy of Sciences), Jinan, Shandong, 250100, China}
\end{center}

\vspace{1mm}
\begin{abstract}

Let $G$ be a connected graph with vertex set $V(G)=\{v_1,v_2,\ldots,v_n\}$, and let \(\omega:V(G)\to \mathbb R^+\) be a positive vertex-weight function satisfying \(\omega(v_i)=x_i\) for each \(v_i \in V(G)\). The weighted Kirchhoff index of $G$ is defined by $K(G;x_1,x_2,\ldots,x_n)=\sum_{1\le i<j\le n}x_i x_j r_G(v_i,v_j)$, where $r_G(v_i,v_j)$ denotes the resistance distance between $v_i$ and $v_j$. In this paper, we give a combinatorial interpretation of the weighted Kirchhoff index of an arbitrary connected graph. More precisely, we express $K(G;x_1,x_2,\ldots,x_n)$ in terms of the sums of weights of matchings in an appropriately weighted subdivision graph of $G$, and in the subgraphs obtained from this weighted subdivision graph by deleting the subdivision graphs corresponding to \(2\)-regular subgraphs of $G$. This gives an affirmative answer to a question posed by Li, Li and Yan [Discrete Math. 345 (2022) 113109] concerning a combinatorial explanation of the weighted Kirchhoff index of a general graph by using matchings in weighted subdivision graphs and their subgraphs. As special cases, our formula recovers the known formulas for the weighted Kirchhoff index of trees and unicyclic graphs, as well as the known formula for the ordinary Kirchhoff index of an arbitrary connected graph.

\vspace{1mm}
\end{abstract}

\section{Introduction}

Resistance distance is a fundamental distance function in graph theory arising from electrical network theory. Let \(G\) be a connected graph with vertex set $V(G)$ and edge set $E(G)$. If each edge of \(G\) is regarded as a unit resistor, then the \emph{resistance distance} between two vertices \(u\) and \(v\), denoted by \(r_G(u,v)\), is defined as the effective resistance between \(u\) and \(v\) in the corresponding electrical network. The notion of resistance distance was formally introduced by Klein and Randi\'{c} \cite{djk1}. In contrast to the classical shortest-path distance $d_G(u,v)$, which is the length of a shortest path connecting $u$ and $v$ \cite{djk1}, the resistance distance satisfies $r_{G}(u,v) \leq d_{G}(u,v)$  with equality if and only if $u$ and $v$ are connected by a unique path. In particular, for a tree, the resistance distance between any two vertices coincides with their distance. Based on the resistance distance, the \emph{Kirchhoff index} \cite{djk1} of a connected graph $G$ is defined by
$$
K(G)=\sum_{\{u,v\}\subseteq V(G)} r_G(u,v).
$$
It can be viewed as a resistance-distance analogue of the Wiener index
\[
W(G)=\sum_{\{u,v\}\subseteq V(G)} d_G(u,v).
\]
As an important graph structure descriptor, the Kirchhoff index plays an essential role in the study of QSAR and QSPR in theoretical chemistry \cite{tws}. It has also been applied to the analysis of network connectivity and
network coherence in circuit theory and complex networks \cite{mty}. For more details, the reader is referred to the recent papers \cite{ani,jhu,jzh,wko,wsu,kcd} and the references therein.

In 2007, Chen and Zhang \cite{hch} introduced the \emph{degree-Kirchhoff index} of a connected graph \(G\), defined by
$$
K'(G)=\sum_{\{u,v\}\subseteq V(G)} d_G(u)d_G(v)r_G(u,v),
$$
where \(d_G(u)\) denotes the degree of the vertex \(u\) in \(G\). This index combines resistance distance with local degree information and can be viewed as a resistance-distance analogue of degree-distance-type graph invariants. Some results on the degree Kirchhoff index of graphs can be found in \cite{jhu1,shu,jpa,zta}.

More generally, let \(G\) be a vertex-weighted connected graph with vertex-weight function \(\omega:V(G)\to \mathbb R^+\). Suppose that \(V(G)=\{v_1,v_2,\ldots,v_n\}\) and \(\omega(v_i)=x_i\) for each \(v_i \in V(G)\). In \cite{sli}, Li, Li and Yan introduced a new index for vertex-weighted graphs, called the \emph{weighted Kirchhoff index}, which is defined by
$$
K(G;x_1,x_2,\ldots,x_n)
=
\sum_{1\le i<j\le n} x_i x_j r_G(v_i,v_j).
$$
Interestingly, this definition contains several important indices as special cases. If we set \(x_i=1\) for every \(v_i \in V(G)\), then we obtain the ordinary Kirchhoff index
\[
K(G;1,1,\ldots,1)=K(G).
\]
If we set \(x_i=d_G(v_i)\) for every \(v_i \in V(G)\), then we obtain the degree Kirchhoff index
\[
K(G;d_G(v_1),d_G(v_2),\ldots,d_G(v_n))=K'(G).
\]
These observations show that the weighted Kirchhoff index provides a unified framework for both the ordinary Kirchhoff index and the degree Kirchhoff index. However, computing the weighted Kirchhoff index for a general graph \(G\)
is far from straightforward.

The aim of this paper is to give a combinatorial interpretation of the weighted Kirchhoff index of an arbitrary connected graph. More precisely, we express $K(G;x_1,x_2,\ldots,x_n)$ in terms of the sums of weights of matchings in an appropriately weighted subdivision graph of $G$, and in the subgraphs obtained from this weighted subdivision graph by deleting the subdivision graphs corresponding to \(2\)-regular subgraphs of $G$. This gives an affirmative answer to the question in \cite{sli} of finding a matching interpretation for the weighted Kirchhoff index of general graphs. As special cases, our formula recovers the known formulas for the weighted Kirchhoff index of trees and unicyclic graphs, as well as the known formula for the ordinary Kirchhoff index of an arbitrary connected graph.

\section{Preliminaries}

\subsection{Notation and basic definitions}
Let \(G=(V(G),E(G))\) be a graph with \(|V(G)|=n\). For an edge \(e=uv\in E(G)\), we denote by \(V(e)=\{u,v\}\) the set of its endpoints.  A subset \(M\subseteq E(G)\) is called
a matching of \(G\) if no two edges in \(M\) share a common endpoint, i.e., $V(e)\cap V(f)=\varnothing $ for any two distinct edges $e,f\in M$. For a nonnegative integer \(j\), let
\begin{equation*}
\mathcal M_j(G)=
\{M\subseteq E(G)\mid M\text{ is a matching of }G,\ |M|=j\}.
\end{equation*}
We denote by \(m(G,j)\) the number of matchings of \(G\) with \(j\) edges, that is, $m(G,j)=|\mathcal M_j(G)|.$ By convention, we set $m(G,0)=1$ and $m(G,j)=0$ for $ j<0$ or $j>\left\lfloor \frac{n}{2}\right\rfloor.$ Now let $G^\omega$ be an edge-weighted graph with weight function $\omega:E(G)\to \mathbb R^+$. For a matching \(M\in\mathcal M_j(G)\),  the weight of  $M$ in $G^\omega$ is the product of weights of edges in $M$, that is
\begin{equation*}
\omega(M)=\prod_{e\in M}\omega(e).
\end{equation*}
The weighted matching number of $G^\omega$ with $j$ edges is defined as
\begin{equation*}
m(G^\omega,j)=\sum_{M\in\mathcal M_j(G)}\omega(M)=\sum_{M\in\mathcal M_j(G)}\prod_{e\in M}\omega(e).
\end{equation*}
In particular, when all edge weights are equal to \(1\), this definition reduces to the ordinary matching number \(m(G,j)\).

The \emph{subdivision graph} of $G$, denoted by $S(G)$, is the graph obtained from $G$ by inserting an additional vertex $e^*$ into each edge $e$ of $G$, see Fig. \ref{Fig1}(b).  In this way,  we can represent $V(S(G))= V(G) \cup \{e^* | e \in E(G) \}$ and $E(S(G))= \{ue^*,ve^* \mid e=uv \in E(G)\}$. For convenience, we refer to a vertex $v_i \in V(G)$ and a vertex $e^* \in \{e^* \mid e \in E(G) \}$  as a \emph{$v$-vertex} and a \emph{$e$-vertex} of $S(G)$, respectively.

\begin{defn}
Let $G$ be a vertex-weighted graph with vertex-weight function $\omega: V(G)\rightarrow \mathbb R^+$ satisfying $\omega(v_i)=x_i$ for each $v_i \in V(G)$. We use $S(G)^{\omega^\ast}$ to denote an edge-weighted
subdivision graph obtained from $S(G)$ by weighting each edge $v_ie^*$ with $\frac{1}{x_i}$ (where $e=v_iv_j \in E(G)$). See Fig. \ref{Fig1}(c) for an example. For convenience, if \(x_i=d_G(v_i)\) for each \(v_i\in V(G)\), then we write \(S(G)^{\omega'}\) for the corresponding edge-weighted subdivision graph. That is, in \(S(G)^{\omega'}\), each edge \(v_i e^\ast\) has weight \(1/d_G(v_i)\), where \(e=v_i v_j\in E(G)\).
\end{defn}

\begin{figure}[htbp]
\centering
\begin{tikzpicture}[
    vnode/.style={circle, fill, inner sep=1.5pt},
    elabel/.style={font=\small, midway},
    xscale=1.5, yscale=1.2
]

\begin{scope}[shift={(0,0)}]
    \node[vnode, label=above:$v_1$] (v1) at (0,0) {};
    \node[vnode, label=left:$v_2$] (v2) at (-1,-1) {};
    \node[vnode, label=below:$v_3$] (v3) at (0,-2) {};   
    \node[vnode, label=right:$v_4$] (v4) at (1,-1) {};  

    \draw (v1) -- (v2) node[elabel, above] {$e_1$};
    \draw (v2) -- (v3) node[elabel, below ] {$e_2$};  
    \draw (v3) -- (v4) node[elabel, below] {$e_3$};
    \draw (v4) -- (v1) node[elabel, above] {$e_4$};
    \draw (v1) -- (v3) node[elabel, left, xshift=3pt] {$e_5$};

    \node at (0,-3) {(a)};
\end{scope}

\begin{scope}[shift={(3.5,0)}]
    \node[vnode, label=above:$v_1$] (v1) at (0,0) {};
    \node[vnode, label=left:$v_2$] (v2) at (-1,-1) {};
    \node[vnode, label=below:$v_3$] (v3) at (0,-2) {};
    \node[vnode, label=right:$v_4$] (v4) at (1,-1) {};
    \node[vnode, label=above:$e^*_1$] (e^*_1) at (-0.5,-0.5) {}; %
    \node[vnode, label=below:$e^*_2$] (e^*_2) at (-0.5,-1.5) {};
    \node[vnode, label=below:$e^*_3$] (e^*_3) at (0.5,-1.5) {};
    \node[vnode, label=above:$e^*_4$] (e^*_4) at (0.5,-0.5) {};
    \node[vnode, label={[xshift=3pt]left:$e^*_5$}] (e5star) at (0,-1) {};
    \draw (v1) -- (v2) node[elabel, above] {};
    \draw (v2) -- (v3) node[elabel, below ] {};  
    \draw (v3) -- (v4) node[elabel, below] {};
    \draw (v4) -- (v1) node[elabel, above] {};
    \draw (v1) -- (v3) node[elabel, left, xshift=3pt] {};

    \node at (0,-3) {(b)};
\end{scope}

\begin{scope}[shift={(7,0)}]
    \node[vnode, label=above:$v_1$] (v1) at (0,0) {};
    \node[vnode, label=left:$v_2$] (v2) at (-1,-1) {};
    \node[vnode, label=below:$v_3$] (v3) at (0,-2) {};
    \node[vnode, label=right:$v_4$] (v4) at (1,-1) {};
    \node[vnode, label=above:$e^*_1$] (e^*_1) at (-0.5,-0.5) {}; %
    \node[vnode, label=below:$e^*_2$] (e^*_2) at (-0.5,-1.5) {};
    \node[vnode, label=below:$e^*_3$] (e^*_3) at (0.5,-1.5) {};
    \node[vnode, label=above:$e^*_4$] (e^*_4) at (0.5,-0.5) {};
    \node[vnode, label={[xshift=3pt]left:$e^*_5$}] (e^*_5) at (0,-1) {};
    \draw (v1) -- (e^*_1) node[pos=0.5, elabel, above] {$\scriptstyle \frac{1}{x_1}$};
    \draw (v2) -- (e^*_1) node[pos=0.5, elabel, above] {$\scriptstyle \frac{1}{x_2}$};
    \draw (v2) -- (e^*_2) node[pos=0.5, elabel, below] {$\scriptstyle \frac{1}{x_2}$};
    \draw (v3) -- (e^*_2) node[pos=0.5, elabel, below] {$\scriptstyle \frac{1}{x_3}$};
    \draw (v3) -- (e^*_3) node[pos=0.5, elabel, below] {$\scriptstyle \frac{1}{x_3}$};
    \draw (v4) -- (e^*_3) node[pos=0.5, elabel, below] {$\scriptstyle \frac{1}{x_4}$};
    \draw (v4) -- (e^*_4) node[pos=0.5, elabel, above] {$\scriptstyle \frac{1}{x_4}$};
    \draw (v1) -- (e^*_4) node[pos=0.5, elabel, above] {$\scriptstyle \frac{1}{x_1}$};
    \draw (v1) -- (e^*_5) node[pos=0.5, elabel, left, xshift=3pt] {$\scriptstyle \frac{1}{x_1}$};
    \draw (v3) -- (e^*_5) node[pos=0.5, elabel, left, xshift=3pt] {$\scriptstyle \frac{1}{x_3}$};

    \node at (0,-3) {(c)};
\end{scope}

\end{tikzpicture}
\caption{(a) The graph \(G\), (b) The subdivision graph \(S(G)\), (c) The edge-weighted subdivision graph \(S(G)^{\omega^\ast}\).} \label{Fig1}

\end{figure}
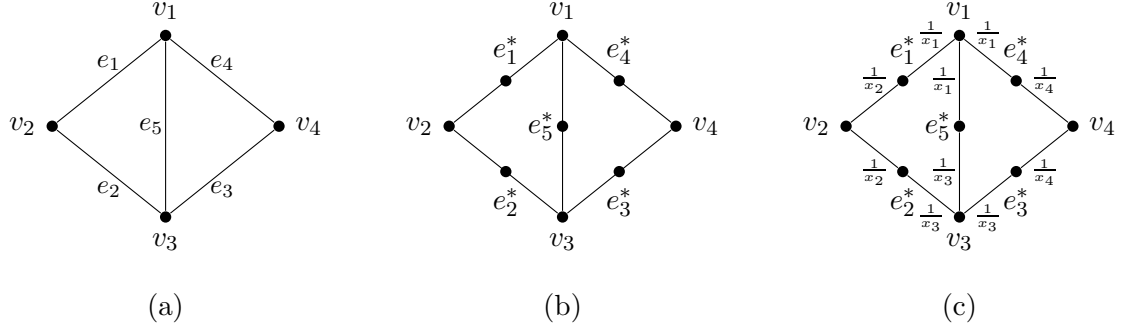
Let \(G\) be a connected graph. If \(H\) is a subgraph of \(G\), we use \(G-H\) to denote the graph obtained from \(G\) by deleting all vertices of \(H\), together with their incident edges. Denote by \(\mathcal C(G)\) the set of all nonempty \(2\)-regular subgraphs of \(G\), equivalently, the set of all nonempty disjoint unions of cycles in \(G\). For \(C\in\mathcal C(G)\), let \(p(C)\) be the number of cycles of \(C\). Since \(C\) is \(2\)-regular, \(|E(C)|=|V(C)|\); we denote this common number by \(|C|\).

We use \(\tau(G)\) to denote the number of \emph{spanning trees} of \(G\).  Let $A_G$ denote the adjacency matrix of $G$, and let $D_G$ denote the diagonal matrix of vertex degrees of $G$. The matrix $L_G=D_G-A_G$ is called the \emph{Laplacian matrix} of $G$. For a vertex $u \in V(G)$, let $L_G(u)$ denote the principal submatrix of $L_G$ obtained by deleting the row and column corresponding to $u$. For two distinct vertices \(u,v\in V(G)\), let \(L_G(u,v)\) denote the principal submatrix of \(L_G\) obtained by deleting the rows and columns corresponding to \(u\) and \(v\). Bapat, Gutman and Xiao \cite{rbb2} gave the
following determinantal formula for resistance distances.
\begin{lem}\cite{rbb2}\label{lem2.3}
Let \(G\) be a connected graph. For two distinct vertices $u, v \in V(G)$, we have
$$r_G(u,v)=\frac{\det L_G(u,v)}{\det L_G(u)}=\frac{\det L_G(u,v)}{\tau(G)}.$$
\end{lem}

\subsection{Combinatorial explanation of the ordinary Kirchhoff index of graphs}
In this subsection, we recall several known results concerning combinatorial interpretations of the Kirchhoff index of graphs via matchings in the subdivision. For a tree $T$,  clearly, the Wiener and Kirchhoff indices of $T$ coincide. In 2006, Yan and Yeh \cite{wya1} showed that
\begin{thm}\cite{wya1}\label{tm2.4}
For a tree \(T\) with \(n\) vertices, the Kirchhoff index (Wiener index) of $T$ can be expressed as
\begin{equation*}
W(T)=K(T)=m(S(T),n-2).
\end{equation*}
\end{thm}
Later, in 2021, Chen and Yan \cite{ych} gave a combinatorial interpretation of the Kirchhoff index for unicyclic graphs, as follows.

\begin{thm}\cite{ych}\label{tm2.5}
Let \(G\) be an arbitrary connected unicyclic graph with \(n\) vertices, and suppose that \(G\) has a cycle \(C_k\) with \(k\) vertices. Then the Kirchhoff index  of \(G\) can be expressed as
\begin{equation*}
K(G)=\begin{cases}\displaystyle\frac{m(S(G),n-2)-2m(S(G)-S(C_k),n-k-2)}{k},
& \text{if } 3\le k\le n-2, \\[3mm]
\displaystyle\frac{m(S(G),n-2)}{k},
& \text{if } k=n-1 \text{ or } k=n.
\end{cases}
\end{equation*}
\end{thm}
In this direction, Que and Chen \cite{lqu} generalized the above results to any connected simple graph.

\begin{thm}\cite{lqu}\label{tm2.6}
Let \(G\) be a connected graph with \(n\) vertices. Then
\begin{equation*}
K(G)=\frac{1}{\tau(G)} \Bigg[m(S(G),n-2)+\displaystyle\sum_{C\in\mathcal C(G)}(-2)^{p(C)}m(S(G)-S(C),n-2-|C|)\Bigg].
\end{equation*}
\end{thm}

Note that Theorem \ref{tm2.6} reduces to Theorem \ref{tm2.4} when \(G\) is a tree, and to Theorem \ref{tm2.5} when \(G\) is a unicyclic graph with cycle \(C_k\).

\subsection{Combinatorial explanation of the weighted Kirchhoff index of graphs}
Note that the weighted Kirchhoff index is a natural generalization of the ordinary Kirchhoff index and the degree Kirchhoff index. It is therefore natural to ask whether the combinatorial interpretations of the Kirchhoff index in terms of matchings can be extended to the weighted Kirchhoff index of graphs. As an important special case, Yan and Li~\cite{wya2} obtained a matching formula for the degree Kirchhoff index of trees.
\begin{thm}\cite{wya2}\label{tm2.7}
Let \(T\) be a tree with vertex set $V(T)= \{v_1,v_2,\ldots,v_n\}$. Then the degree Kirchhoff index of \(T\) satisfies
\begin{equation*}
K'(T)=m\left(S(T)^{\omega'},n-2\right) \prod_{i=1}^{n}d_T(v_i),
\end{equation*}
\end{thm}
Recently, Li, Li and Yan \cite{sli} generalized the above result to weighted Kirchhoff index of trees.
\begin{thm}\cite{sli}\label{tm2.8}
Let \(T\) be a vertex-weighted tree with vertex set $V(T)= \{v_1,v_2,\ldots,v_n\}$ satisfying $\omega(v_i)=x_i$ for each $v_i \in V(T)$. The weighted Kirchhoff  index (weighted Wiener index) of $T$ can be expressed as
\begin{equation*}
K(T;x_1,x_2,\ldots,x_n)=m\left(S(T)^{\omega^\ast},n-2\right)\prod_{i=1}^{n}x_i,
\end{equation*}
\end{thm}
Clearly, Theorem \ref{tm2.8} generalizes Theorems \ref{tm2.4} and \ref{tm2.7}.  They further gave a combinatorial explanation of the weighted Kirchhoff index of  unicyclic graphs.

\begin{thm}\cite{sli}\label{tm2.9}
Let \(G\) be a connected vertex-weighted unicyclic graph with $V(G)= \{v_1,v_2,\ldots,v_n\}$ satisfying $\omega(v_i)=x_i$ for each $v_i \in V(G)$, and suppose that $G$ has a cycle $C_k$ with $k$ vertices. Then the weighted Kirchhoff index $K(G;x_1,x_2,\ldots,x_n)$ can be expressed as
\begin{equation*}\label{eq:weighted-unicyclic}
\begin{aligned}
&K(G;x_1,x_2,\ldots,x_n)
=\\
&\begin{cases}
\displaystyle
\frac{1}{k}
\left[
\begin{aligned}
 m\left(S(G)^{\omega^\ast},n-2\right)\prod_{i=1}^{n}x_i-2m\left(S(G)^{\omega^\ast}-C_{2k},n-k-2\right)
\prod_{v_i\notin V(C_k)}x_i
\end{aligned}
\right],
& \text{if } 3\le k\le n-2, \\[7mm]
\displaystyle
\frac{1}{k}
m\left(S(G)^{\omega^\ast},n-2\right)\prod_{i=1}^{n}x_i,
& \text{if } k=n-1 \text{ or } k=n.
\end{cases}
\end{aligned}
\end{equation*}
\end{thm}
In addition,  they  proposed the following natural question.
\begin{ques}\cite{sli}
Give a combinatorial explanation of the weighted Kirchhoff index of a general graph G by using the matchings of $S(G)^{\omega^\ast}$ and its subgraphs.
\end{ques}
In this paper, we answer this question affirmatively as follows and the proof will be given in Section \ref{section3}.
\begin{thm}\label{tm2.10}
Let $G$ be a connected vertex-weighted graph with vertex set $V(G)=\{v_1,v_2,\ldots,v_n\}$ satisfying $\omega(v_i)=x_i$ for each $v_i \in V(G)$.  Then the weighted Kirchhoff index of $G$ satisfies
\begin{equation}
\begin{aligned}K(G;x_1,x_2,\ldots,x_n)&= \frac{1}{\tau(G)}\Bigg[m(S(G)^{\omega^\ast},n-2)\prod_{i=1}^n x_i \\
&\qquad + \sum_{C\in\mathcal C(G)}(-2)^{p(C)}m(S(G)^{\omega^\ast}-S(C),n-2-|C|)\prod_{v_i\notin V(C)}x_i\Bigg].
\end{aligned}
\end{equation}
\end{thm}
It can be seen that Theorem  \ref{tm2.10} is a generalization of Theorems \ref{tm2.5}-\ref{tm2.9}. On the other hand, by taking \(x_i=d_G(v_i)\), we obtain the following corollary.
\begin{cor}
Let $G$ be a connected graph with vertex set $V(G)= \{v_1,v_2,\ldots,v_n\}$.  Then the degree Kirchhoff index of $G$ satisfies
\begin{equation*}
\begin{aligned}K'(G)&= \frac{1}{\tau(G)}\Bigg[m(S(G)^{\omega'},n-2)\prod_{i=1}^n d_G(v_i) \\
&\qquad + \sum_{C\in\mathcal C(G)}(-2)^{p(C)}m(S(G)^{\omega'}-S(C),n-2-|C|)\prod_{v_i\notin V(C)}d_G(v_i)\Bigg].
\end{aligned}
\end{equation*}
\end{cor}

\section{Proof of Theorem \ref{tm2.10}}\label{section3}
In this section, we give the proof of Theorem \ref{tm2.10}.
\begin{proof}
We first prove the following auxiliary identity. For any two distinct
vertices $v_i,v_j\in V(G)$,
\begin{align}\label{eq3.1}
\tau(G)r_G(v_i,v_j) &=m(S(G)-v_i-v_j,n-2) \notag \\
&+\sum_{\substack{C\in\mathcal C(G)\\ v_i,v_j\notin V(C)}}(-2)^{p(C)}m(S(G)-S(C)-v_i-v_j,n-2-|C|).
\end{align}
For convenience, set \(X=V(G)\setminus\{v_i,v_j\}\). Then \(L_G[X]\) denotes the principal submatrix of \(L_G\) indexed by \(X\); equivalently, \(L_G[X]=L_G(v_i,v_j)\). By Lemma \ref{lem2.3}, we have
\begin{equation}\label{eq3.2}
\det L_G[X]=\tau(G)r_G(v_i,v_j).
\end{equation}

So by Eq.~\eqref{eq3.2}, in order to prove Eq.~\eqref{eq3.1}, it suffices to prove the following claim.
\begin{cla}\label{claim1}
\begin{equation*}
\det L_G[X] =m(S(G)-v_i-v_j,n-2)+\sum_{\substack{C\in\mathcal C(G)\\ V(C)\subseteq X}}(-2)^{p(C)}m(S(G)-S(C)-v_i-v_j,n-2-|C|).
\end{equation*}
\begin{equation*}
\end{equation*}
\end{cla}
\emph{Proof of Claim 1.}  By the Leibniz formula for the determinant, we have
\begin{equation}\label{eq3.3}
\det L_G[X]=\sum_{\pi\in\mathfrak S_X}\operatorname{sgn}(\pi)\prod_{x\in X} (L_G)_{x,\pi(x)}.
\end{equation}
Here \(\mathfrak S_X\) denotes the set of all permutations of \(X\). For \(\pi\in\mathfrak S_X\), \(\operatorname{sgn}(\pi)\) denotes the sign of the permutation \(\pi\).

For each \(\pi\in\mathfrak S_X\), let \(\mathcal C(\pi)\) be the contribution of \(\pi\) to the Leibniz expansion of \(\det L_G[X]\), namely
\begin{equation*}
\mathcal C(\pi)=\operatorname{sgn}(\pi)\prod_{x\in X}(L_G)_{x,\pi(x)}.
\end{equation*}
Therefore, we have
\begin{equation}\label{eq3.4}
\det L_G[X]=\sum_{\pi\in\mathfrak S_X}\mathcal C(\pi).
\end{equation}
Note that each permutation $\pi$ can be decomposed into disjoint permutation cycles, and each permutation cycle is of one of the following three types:
\begin{itemize}
    \item a fixed point: \(\pi(u)=u\);
    \item a transposition \((u\,v)\): \(\pi(u)=v,\ \pi(v)=u\);
    \item a permutation cycle of length at least \(3\):
    \[
    u_1\mapsto u_2\mapsto\cdots\mapsto u_\ell\mapsto u_1,
    \qquad \ell\ge3.
    \]
\end{itemize}

For a fixed point \(u\) of \(\pi\) (i.e., $\pi(u)=u$), the factor corresponding to \(u\) in the product $\prod_{x\in X}(L_G)_{x,\pi(x)}$ is $d_G(u).$ We expand this diagonal entry as
\begin{equation}\label{eq3.5}
d_G(u)=\sum_{\substack{e\in E(G)\\ u\in V(e)}}1.
\end{equation}
Thus, after expanding all diagonal entries, a fixed point \(u\) may be regarded as choosing one edge incident with \(u\), each such choice having weight \(1\). This way of expanding the diagonal entries will play a crucial role in the following proof.

For a transposition \((u\,v)\) of \(\pi\), the corresponding factors are \((L_G)_{uv}\) and \((L_G)_{vu}\). If \(uv\notin E(G)\), then $(L_G)_{uv}=(L_G)_{vu}=0,$ and hence the corresponding term is zero. Thus we may only consider the case \(uv\in E(G)\).

Suppose \(uv\in E(G)\). Let \(\pi_0\) be the permutation obtained from
\(\pi\) by replacing the transposition \((u\,v)\) with two fixed points
\(u\) and \(v\), while leaving all other cycles unchanged. Thus  $\pi=(u\,v)\circ \pi_0.$

Let \(A\) denote the product of all factors coming from vertices other
than \(u\) and \(v\), i.e.,
\[
A=\prod_{\substack{x\in X\\ x\ne u,v}}(L_G)_{x,\pi_0(x)}.
\]

If \(u\) and \(v\) form a transposition, then, after fixing all choices arising from the expansions of the diagonal entries corresponding to the other fixed points, the associated completely expanded term is
\[
\begin{aligned}
\mathcal C(\pi)
&=\operatorname{sgn}(\pi)\prod_{x\in X}(L_G)_{x,\pi(x)} \\[1mm]
&=\operatorname{sgn}\bigl((u\,v)\circ\pi_0\bigr)(L_G)_{uv}(L_G)_{vu}
\prod_{\substack{x\in X\\x\ne u,v}}(L_G)_{x,\pi_0(x)} \\[1mm]
&=\bigl(-\operatorname{sgn}(\pi_0)\bigr)(-1)(-1)A \\[1mm]
&=-\operatorname{sgn}(\pi_0)A.
\end{aligned}
\]

On the other hand, in the permutation \(\pi_0\), the vertices \(u\) and
\(v\) are fixed. We have
\[
\begin{aligned}
\mathcal C(\pi_0)
&=\operatorname{sgn}(\pi_0)(L_G)_{uu}(L_G)_{vv}
\prod_{\substack{x\in X\\x\ne u,v}}(L_G)_{x,\pi_0(x)} \\[1mm]
&=\operatorname{sgn}(\pi_0)d_G(u)d_G(v)A.
\end{aligned}
\]
By the expansion of the diagonal entries in Eq.~(\ref{eq3.5}), the product \(d_G(u)d_G(v)\) expands into \(d_G(u)d_G(v)\) summands, each having value \(1\). Hence, with all other choices fixed, every corresponding completely expanded term has contribution \(\operatorname{sgn}(\pi_0)A\). Since \(uv\in E(G)\), there is a unique such term in which both \(u\) and \(v\) choose the edge \(uv\). The contribution of this particular term is
\[
\begin{aligned}
\frac{1}{d_G(u)d_G(v)}\mathcal C(\pi_0)
&=\frac{1}{d_G(u)d_G(v)}
\operatorname{sgn}(\pi_0)d_G(u)d_G(v)A \\[1mm]
&=\operatorname{sgn}(\pi_0)A.
\end{aligned}
\]
By the previous computation, this is exactly \(-\mathcal C(\pi)\). Thus, replacing a transposition \((u\,v)\) by two fixed points \(u\) and \(v\), both choosing the edge \(uv\), reverses the sign of the corresponding expanded term.

We now define a global sign-reversing pairing of the expanded terms. Fix a linear order on the vertices, and use it to order the edges lexicographically. Call an edge \(uv\) \emph{bad} in an expanded term if one of the following two conditions holds:
\begin{enumerate}
    \item[(1).] \((u\,v)\) is a transposition of the corresponding permutation;
    \item[(2).] \(u\) and \(v\) are both fixed points of the corresponding permutation, and both choose the edge \(uv\).
\end{enumerate}
For each expanded term containing at least one bad edge, choose the smallest bad edge \(uv\). If \(uv\) appears as a transposition, replace it by two fixed points \(u\) and \(v\), both choosing the edge \(uv\). If \(uv\) appears as two fixed points both choosing \(uv\), replace them by the transposition \((u\,v)\). All other parts are left unchanged.

It is not hard to see that this operation changes only the status of the edge \(uv\), from a transposition to two fixed points both choosing \(uv\), or conversely. No other bad edge is changed or created. Hence the set of bad edges is preserved, and the smallest bad edge is still \(uv\). Therefore applying the operation twice gives back the original expanded term, and no expanded term is paired with itself. Moreover, the two paired expanded terms have opposite signs. Hence all expanded terms containing at least one bad edge cancel in pairs. Consequently, in computing the net contribution, it remains to consider expanded terms with no bad edge. Equivalently, the remaining terms have no transpositions, and no two fixed points choose the same edge of \(G\).

There are two possible cases for the underlying permutation of such an expanded term. Either it has no permutation cycle of length at least \(3\), or it has at least one permutation cycle of length at least \(3\). In the first case, we set \(C=\varnothing\). In the second case, suppose that
\[
u_1\mapsto u_2\mapsto\cdots\mapsto u_\ell\mapsto u_1,
\qquad \ell\ge3,
\]
is such a permutation cycle. The factor contributed by this permutation
cycle is nonzero only if $u_1u_2,\ u_2u_3,\ \ldots,\ u_{\ell-1}u_\ell,\ u_\ell u_1$ are all edges of \(G\). Hence the vertices \(u_1,\ldots,u_\ell\) form a cycle of \(G\). Since distinct permutation cycles are vertex-disjoint, all permutation cycles of length at least \(3\) with nonzero factors together determine a nonempty \(2\)-regular subgraph \(C\in\mathcal C(G)\) such that \(V(C)\subseteq X,\) or equivalently, \(v_i,v_j\notin V(C)\). Conversely, each such \(2\)-regular subgraph \(C\), together with a choice of an orientation for each cycle of \(C\), determines the corresponding permutation cycles of length at least 3.

We now consider separately the case \(C=\varnothing\) and the case \(C\neq\varnothing\) (i.e., \(C\in\mathcal C(G)\) with \(V(C)\subseteq X\)).

If \(C=\varnothing\), then, after all transpositions have been cancelled, the only remaining permutation is the identity permutation on \(X\), denoted by \(\operatorname{id}_X\). Thus every vertex of \(X\) is a fixed point. Its contribution is
\begin{align}
\mathcal C(\mathrm{id_X}) &=\operatorname{sgn}(\mathrm{id_X})\prod_{x\in X}(L_G)_{x,x} =\prod_{x\in X}d_G(x) \notag \\
&=\prod_{x\in X}\Bigg(\sum_{\substack{e\in E(G)\\ x\in V(e)}}1\Bigg) \notag \\
&=\sum_{\mathcal E_X}1, \notag
\end{align}
where
\[
\mathcal E_X=
\left\{
(e_x)_{x\in X}: e_x\in E(G)\text{ and }x\in V(e_x)\text{ for every }x\in X
\right\}.
\]

By the cancellation with transpositions, all summands in which two vertices \(x,y\in X\) choose the same edge \(xy\) are cancelled. Therefore, the remaining elements in $\mathcal E_X$ are precisely those satisfying
\[
e_x\neq e_y\qquad \text{for all distinct }x,y\in X.
\]

Now we consider the graph $S(G)-v_i-v_j$ obtained from the subdivision graph $S(G)$ by deleting vertices $v_i$ and $v_j$. For each \(x\in X\), if \(e_x\) is the edge chosen by \(x\), we choose the subdivision edge $x e_x^\ast$. Since the vertices \(x\in X\) are distinct, no two chosen subdivision edges share a \(v\)-type vertex. Similarly, the chosen edges \(e_x\) are distinct, and no two chosen subdivision edges share an \(e\)-type vertex. Hence these subdivision edges form a matching of size $ |X|=n-2$ in \(S(G)-v_i-v_j\).

Conversely, since \(S(G)-v_i-v_j\) has exactly $(n-2)$ \(v\)-type vertices and every subdivision edge is incident with exactly one \(v\)-type vertex, every matching of size \(n-2\) in \(S(G)-v_i-v_j\) saturates all \(v\)-type vertices in \(X\). It therefore uniquely determines a remaining element \((e_x)_{x\in X}\).

Consequently, the net contribution of the case \(C=\varnothing\) is
\begin{equation}\label{eq3.6}
m(S(G)-v_i-v_j,n-2).
\end{equation}

If \(C\neq\varnothing\), fix a nonempty \(2\)-regular subgraph \(C\) with \(V(C)\subseteq X\). For each choice of orientations of the cycles of \(C\), let \(\pi\) be the permutation whose permutation cycles are exactly these
oriented cycles on \(V(C)\), and note that the remaining vertices in \(X\setminus V(C)\) are fixed points (as transpositions have already been cancelled).  Each oriented cycle of length \(\ell\) contributes
\begin{equation*}
\underbrace{(-1)^{\ell-1}}_{\text{sign of }\ell\text{-cycle}} \cdot
\underbrace{(-1)^\ell}_{\text{\(L_G\)-entries along edges}} = -1.
\end{equation*}
Then, we have
\[
\begin{aligned}
\mathcal C(\pi)&=\operatorname{sgn}(\pi)\prod_{x\in X}(L_G)_{x,\pi(x)}\\
&=(-1)^{p(C)}\prod_{x\in X\setminus V(C)}(L_G)_{x,x}\\
&=(-1)^{p(C)}\prod_{x\in X\setminus V(C)}d_G(x).
\end{aligned}
\]
Since each undirected cycle of \(C\) has two possible orientations, summing over all \(2^{p(C)}\) choices of orientations gives
\[
\begin{aligned}
\sum_{\text{orientations of }C}\mathcal C(\pi)
&=2^{p(C)}(-1)^{p(C)}\prod_{x\in X\setminus V(C)}d_G(x)\\
&=(-2)^{p(C)}\prod_{x\in X\setminus V(C)}d_G(x).
\end{aligned}
\]
For a vertex \(x\in X\setminus V(C)\), there is no edge of \(E(C)\) incident with \(x\). Hence
\[
d_G(x)=\sum_{\substack{e\in E(G)\\  x \in V(e)}}1=\sum_{\substack{e\in E(G)\setminus E(C)\\ x \in V(e)}}1.
\]
Therefore,
\[
\prod_{x\in X\setminus V(C)}d_G(x)=\prod_{x\in X\setminus V(C)}\left(\sum_{\substack{e\in E(G)\setminus E(C)\\ x\in V(e)}}1\right)=\sum_{\mathcal E_{X,C}}1,
\]
where
\[
\mathcal E_{X,C}=\left\{(e_x)_{x\in X\setminus V(C)}:e_x\in E(G)\setminus E(C)\text{ and }x\in V(e_x)\text{ for every }x\in X\setminus V(C)\right\}.
\]

Similarly, by the cancellation with transpositions, all elements in $\mathcal E_{X,C}$ in which two distinct fixed points choose the same edge are removed. So the remaining elements in \(\mathcal E_{X,C}\) are precisely those for which no two distinct vertices choose the same edge, that is,
\[
e_x\neq e_y
\qquad
\text{for all distinct }x,y\in X\setminus V(C).
\]

We now translate these remaining elements into matchings of the graph \(H=S(G)-S(C)-v_i-v_j\). The \(v\)-type vertices of \(H\) are precisely the vertices in \(X\setminus V(C)\). Hence the number of \(v\)-type vertices in \(H\) is
\[
|X\setminus V(C)|=|X|-|C|=n-2-|V(C)|.
\]

Given a remaining element \((e_x)_{x\in X\setminus V(C)}\), choose the subdivision edge \(xe_x^\ast\) for each \(x\in X\setminus V(C)\). Here \(e_x\in E(G)\setminus E(C)\), so the $e$-type vertex \(e_x^\ast\)
remains in \(H\), and hence \(xe_x^\ast\) is indeed an edge of \(H\). The vertices \(x\in X\setminus V(C)\) are distinct, so no two chosen subdivision edges share a \(v\)-type vertex. Moreover, \(e_x\neq e_y\)
whenever \(x\neq y\), and hence no two chosen subdivision edges share an \(e\)-type vertex. Therefore, these chosen subdivision edges form a matching of size \(n-2-|C|\) in \(H\).

Conversely, let \(M\) be a matching of size \(n-2-|C|\) in \(H\). Since \(H\) has exactly \(n-2-|C|\) \(v\)-type vertices, and every edge of \(H\) is incident with exactly one \(v\)-type vertex, the matching
\(M\) must saturate all \(v\)-type vertices of \(H\), namely all vertices in \(X\setminus V(C)\). Thus, for each \(x\in X\setminus V(C)\), there is a unique edge of \(M\) incident with \(x\). In the subdivision graph, this edge must be of the form \(x e^*\), where \(e^*\) is the vertex corresponding to an original edge \(e\in E(G)\setminus E(C)\). We then define \(e_x=e\). Since \(M\) is a matching, distinct vertices \(x\) yield distinct edge-vertices \(e^*\), and hence \(e_x\neq e_y\) whenever \(x\neq y\). Therefore \(M\) uniquely determines a remaining element $(e_x)_{x\in X\setminus V(C)}\in \mathcal E_{X,C}.$

Therefore, for this fixed nonempty \(2\)-regular subgraph \(C\), the net contribution is
\begin{equation}\label{eq3.7}
(-2)^{p(C)}
m(S(G)-S(C)-v_i-v_j,n-2-|C|).
\end{equation}
By the Leibniz expansion in Eq.~\eqref{eq3.3}, together with Eq.~\eqref{eq3.4}, the determinant \(\det L_G[X]\) is obtained by summing the contributions of all possible permutation structures. Combining the net contributions computed in Eqs.~\eqref{eq3.6} and \eqref{eq3.7}, and summing over all \(C=\varnothing\) and over all \(C\in\mathcal C(G)\) with
\(V(C)\subseteq X\), we get
\[
\det L_G[X]=m(S(G)-v_i-v_j,n-2)+\sum_{\substack{C\in\mathcal C(G)\\ v_i,v_j\notin V(C)}}(-2)^{p(C)}m(S(G)-S(C)-v_i-v_j,n-2-|C|).
\]
Thus, Claim \ref{claim1} is proved.

Now let
\begin{align}
\Phi(x_1,\ldots,x_n)=&m(S(G)^{\omega^\ast},n-2)\prod_{i=1}^n x_i  \notag \\
&+\sum_{C\in\mathcal C(G)}(-2)^{p(C)}m(S(G)^{\omega^\ast}-S(C),n-2-|C|)\prod_{v_i\notin V(C)}x_i.\notag
\end{align}

To give the main result, we aim to prove that
\begin{equation}\label{eq3.8}
\Phi(x_1,\ldots,x_n)=\tau(G)\sum_{1\leq i<j\leq n}x_ix_jr_G(v_i,v_j).
\end{equation}
According to the definition of $m(S(G)^{\omega^\ast},k)$, we know that both sides of Eq. (\ref{eq3.8}) are homogeneous square-free polynomials of degree two in \({x_1,\ldots,x_n}\). Hence it is enough to prove that the coefficients of \(x_ix_j\) on both sides are equal for every \(1\le i<j\le n\).

By the definition of  $S(G)^{\omega^\ast}$, for any $1 \leq i < j \leq n$, we have
\begin{equation*}
m(S(G)^{\omega^\ast}-v_i-v_j,n-2)=m(S(G)-v_i-v_j,n-2)\prod_{\substack{k=1\\ k\neq i,j}}^n\frac{1}{x_k}.
\end{equation*}
Note that
\begin{equation*}
m(S(G)^{\omega^\ast},n-2)=\sum_{1 \leq p< q \leq n} m(S(G)^{\omega^\ast}-v_p-v_q,n-2).
\end{equation*}
Therefore, the coefficient of $x_ix_j$ in $m(S(G)^{\omega^\ast},n-2)\prod_{i=1}^n x_i$ is
\begin{equation}\label{eq3.9}
m(S(G)-v_i-v_j,n-2).
\end{equation}

Now fix \(C\in\mathcal C(G)\). If \(v_i\in V(C)\) or \(v_j\in V(C)\), then the factor $\prod_{v_t\notin V(C)}x_t$ does not contain both \(x_i\) and \(x_j\), so this term contributes nothing to the coefficient of \(x_ix_j\).

Assume now that \(v_i,v_j\notin V(C)\). Similarly,
\[
\begin{aligned}
&m(S(G)^{\omega^\ast}-S(C)-v_i-v_j,n-2-|C|) \\
&\qquad =
m(S(G)-S(C)-v_i-v_j,n-2-|C|)
\prod_{\substack{v_k\notin V(C)\\ k\neq i,j}}\frac{1}{x_k}.
\end{aligned}
\]
Also, we have
\[
m(S(G)^{\omega^\ast}-S(C),n-2-|C|)=\sum_{\substack{1\le p<q\le n\\ v_p,v_q\notin V(C)}}m(S(G)^{\omega^\ast}-S(C)-v_p-v_q,n-2-|C|).
\]
Hence the coefficient of \(x_ix_j\) in $m(S(G)^{\omega^\ast}-S(C),n-2-|C|)\prod_{v_i\notin V(C)}x_i$ is
\begin{equation}\label{eq3.10}
m(S(G)-S(C)-v_i-v_j,n-2-|C|).
\end{equation}
Combining the two coefficient computations in Eqs. (\ref{eq3.9}) and (\ref{eq3.10}), we obtain that the coefficient of \(x_i x_j\) in \(\Phi\) is
\[
m(S(G)-v_i-v_j,n-2) + \sum_{\substack{C\in\mathcal C(G)\\ v_i,v_j\notin V(C)}}(-2)^{p(C)}m(S(G)-S(C)-v_i-v_j,n-2-|C|).
\]
By Eq. (\ref{eq3.1}), this coefficient is exactly $\tau(G)r_G(v_i,v_j).$ Hence the coefficients of \(x_i x_j\) on both sides of Eq.~(\ref{eq3.8}) are equal for every \(1\le i<j\le n\). Therefore, Eq.~(\ref{eq3.8}) holds. Consequently,
\[
K(G;x_1,x_2,\ldots,x_n)=\frac{\Phi(x_1,\ldots,x_n)}{\tau(G)}.
\]
This is the desired formula, and the proof is complete.
\end{proof}

\section{Declaration of competing interest}
The authors declare that they have no known competing financial interests or personal relationships that could have appeared to influence the work reported in this paper.
\section{Data availability}
No data was used for the research described in the article.
\section{Acknowledgments}
The second author is supported by the Taishan Scholars Special Project of Shandong Province with the Grant no. tstp202607065 and the National Natural Science Foundation of China with the Grant no. 12171414. The third author is supported by the National Natural Science Foundation of China with the Grant no. 12071194.

\section{Declaration of generative AI and AI-assisted technologies in the manuscript preparation process}
During the preparation of this work, the authors used AI for checking grammatical errors and typos during the revision stage. The authors reviewed and edited the output as needed and take full responsibility for the content of the published article.

\end{CJK}
\end{document}